\documentclass[11pt]{amsart}
\usepackage{a4}
\usepackage{amsmath}
\usepackage{amsfonts}
\usepackage{amsthm}
\usepackage{amssymb}
\usepackage[usenames,dvipsnames,svgnames,table]{xcolor}
\usepackage{graphicx}
\usepackage{tikz}
\usepackage{tkz-graph}
\usepackage{tkz-berge}
\usetikzlibrary{arrows,shapes}
\usepackage{url}
\usepackage{pgf}
\usepackage{tikz}
\usepackage{algorithm}
\usepackage{algorithmic}

\usepackage{caption}
\usepackage{subcaption} 
\usepackage{float}  

\newtheorem{theorem}{Theorem}[section]

\newtheorem{prop}[theorem]{Proposition}

\newtheorem{corollary}[theorem]{Corollary}

\newtheorem{dfn}[theorem]{Definition}

\newtheorem{rmk}[theorem]{Remark}

\def\thm{\textbftheorem}
\def \bp {\begin{prp} \ }
	\def \ep {\end{prp}}
\def \bpm {\begin{prm} \ }
	\def \epm {\end{prm}}
\def \bc {\begin{crl} \ }
	\def \ec {\end{crl}}
\def \thm {\begin{Theorem} \ }
	\def \ethm {\end{Theorem}}
\def \bl {\begin{lem} \ }
	\def \el {\end{lem}}
\def \bd {\begin{defi} \ \rm }
	\def \ed {\end{defi}}
\def \brm {\begin{rmk} \ }
	\def \erm {\end{rmk}}
\def \bxm {\begin{xmp} \ \rm }
	\def \exm {\end{xmp}}
\def \bcj {\begin{conj}}
	\def \ecj {\end{conj}}
\def \nmr {\begin{enumerate}}
	\def \enmr {\end{enumerate}}
\def \tmz {\begin{itemize}}
	\def \etmz {\end{itemize}}

\begin{document}
\title{Connected dom-forcing Sets in Graphs}
\maketitle
\begin{center}
  \author{\bf \sc Susanth P $^{1,2}$\  Charles Dominic $^{3} $ \ Premodkumar K P $^{4} $ \\{\footnotesize $^{1}$ Department of Mathematics},\\{\footnotesize Pookoya Thangal Memorial Government College,\\(Affiliated to University of Calicut)\\ Perinthalmanna, Kerala- 679322, INDIA.}\\{\footnotesize $^{2}$ Department of Mathematics}\\{\footnotesize St.Joseph's  College(Autonomous),\\ Devagiri, Calicut~-~673008, INDIA.} \\{\footnotesize $^{3}$ Flat 3, Stone Bridge House, \\Stonebridge walk, Chelmsford, CM1 1DN,\\ Essex, United Kingdom.} \\{\footnotesize $^4$ Department of Mathematics},\\{\footnotesize Govt. College Malappuram ,\\ Kerala- 676509, INDIA.} \\ {\footnotesize $^{1,2}$ E-mail: $psusanth@gmail.com $}\\ { \hspace{1cm}\footnotesize $^{3}$ E-mail: $charlesdominicpu@gmail.com$}\\
 {\footnotesize $^{4}$ E-mail: $pramod674@gmail.com  $} } 
\end{center}

\begin{abstract}
	In a graph $G$, a dominating set $D_{f}\subseteq V(G)$  is called a \emph{dom-forcing set} if the sub-graph induced by $\langle D_{f}  \rangle$  must form a zero forcing set. The minimum cardinality of such a set is known as the dom-forcing number of the graph $G$, denoted by $F_{d}(G)$. A connected dom-forcing forcing set  of a graph $G$, is a dom-forcing set of $G$ that induces a sub graph of $G$ which is connected. The connected dom-forcing number of G, $F_{cd}(G)$, is the minimum size of a connected dom-forcing set.  This study delves into the concept of the connected dom-forcing number $F_{cd}(G)$, examining its properties and characteristics. Furthermore, it seeks to accurately determine $F_{cd}(G)$ for several well-known graphs and their graph products.
	\end{abstract}

\title{}

\section{Introduction}

A zero forcing set contains certain vertices in a graph that are initially colored black after applying the following color-change rule. Color-change rule: Let $G$ be a graph with each vertex is colored either white or black, $u$ be a black vertex of $G$, and exactly one neighbor $v$ of $u$ be white. Then change the color of $v$ to black. When this rule is applied, we say $u$ forces $v$, and write $u \rightarrow v$ see \cite{aim}.
 A zero forcing set ( or $\mathcal{ZFS}$ for brevity) of a graph $G$ is a subset $Z$ of vertices such that if initially the vertices in $Z$ are colored black and remaining vertices are colored white, the entire graph $G$ may be colored black by repeatedly applying the color-change rule. The minimum $|Z|$ over all zero forcing sets $Z \subset G$ is called the zero forcing number and is denoted by $Z(G)$.Any zero forcing set of order $Z(G)$ is called a minimum zero forcing set \cite{aim}.

A connected zero forcing set ( or $\mathcal{CZFS}$ for brevity) of a graph $G$, is a zero forcing set of $G$ that induces a sub graph of $G$ which is connected. The connected zero forcing number of $G$, $Z_c(G)$, is the minimum size of connected zero forcing sets. Any connected zero forcing set of order $Z_c(G)$ is called a minimum connected zero forcing set \cite{czero}.

A subset $D$ of vertex set $V(G) $  in a graph $G=(V, E)$ is referred a dominating set if every vertex $v\in V(G)$ is either an element of $D$ or as at least one neighbor in $D$ see\cite{dom}.

If a sub graph formed by the vertices of a dominating set in a graph $G$ is connected then the set is called a connected dominating set.
A minimum connected dominating set of a graph $G$ is a connected dominating set with the smallest possible cardinality among all connected dominating sets of $G$. The connected domination number of $G$ is the number of vertices in the minimum connected dominating set, and is denoted by $\gamma_c(G)$ \cite{cdom}.

A set $D_{f}\subseteq V$ of vertices is called a \emph{dom-forcing set} if it satisfies the following two conditions.\\
i)  $ D_{f}  $ must form a  dominating set.\\
ii) $ D_{f}  $ must form a  zero forcing set.\\
The minimum cardinality of such a set is called the dom-forcing number of the graph $G$ and is denoted by $F_{d}(G)$ \cite{df1}.

This study integrates the concepts of connected domination and connected zero forcing to introduce a new graph-theoretic parameter: the connected dom-forcing set. In this article we consider only connected graphs. 
\section{motivation}
Zero forcing and domination are two independently rich and fundamental concepts in graph theory that each have significant applications to problems in network theory, quantum control of systems, power grid surveillance, and information spreading. Domination reflects the notion of local influence—forcing all vertices to be either in a chosen set or within distance one of it—while zero forcing describes how influence spreads throughout a network from a minimal starting set.

The concept of a dom-forcing set comes naturally as a synthesis of these two concepts. It's a vertex set that not only has immediate control over the graph (by domination) but also sets off a process that ends up forcing the whole network (by zero forcing). By combining these ideas, dom-forcing sets form a more robust tool to study how influence can be gained and extended within intricate systems.

The addition of connected dom-forcing sets brings in yet another essential aspect—connectivity. In actual networks like communication infrastructures, biological systems, and social networks, connectedness of the initial control set is generally necessary for resilience, coordination, and effective transmission of information. A connected dom-forcing set guarantees that influence is not merely exhaustive but also structurally connected and resilient against interference.

Investigating dom-forcing sets and their related variants opens a rich class of mathematical problems, such as finding bounds, describing extremal instances, and studying their relation to traditional graph invariants like the domination number and zero forcing number. Such studies also enable practical applications in which efficient control needs to reconcile immediate influence, dynamic spread, and structural robustness.

In conclusion, the study of dom-forcing sets is driven by its theoretical value as well as by its prospects for solving practical problems in network control more exactly and effectively.
\section{Connected dom-forcing set}
In this section, we explore the idea of a connected dom-forcing set. Additionally, we do a comparison between the dom-forcing set and the connected dom-forcing set in the context of specific types of graphs.
\begin{dfn}
    A set $CD_{f}\subseteq V$ of vertices is called a \emph{connected dom-forcing set} if it satisfies the following two conditions.\\
i)  $\langle CD_{f}  \rangle $ must form a  connected dominating set.\\
ii) $\langle CD_{f}  \rangle $ must form a  connected zero forcing set.\\
\end{dfn}
The minimum cardinality of such a set is called the connected dom-forcing number of the graph $G$ and is denoted by $F_{cd}(G)$. For instance, contemplate the graph $C_5$ illustrated in Figure 1.  
\begin{center}
	\makeatletter
	\begin{figure}[H]
		\centering
  \definecolor{ffffff}{rgb}{1,1,1}
		\begin{tikzpicture}[every node/.style={fill=red!60,circle,inner sep=1pt},
			.style={sibling distance=20mm,nodes={fill=red!45}},
			.style={sibling distance=20mm,nodes={fill=red!30}},
			.style={sibling distance=20mm,nodes={fill=red!25}}, style= thick]
			\draw [line width=1.5pt,color=blue,step=.5cm,] (3,2.5)to[bend right=30](1,1);
			\draw [line width=1.5pt,color=blue,step=.5cm,] (1,1)to[bend right=30](2,-1);
			
			\draw [line width=1.5pt,color=blue,step=.5cm,] (2,-1)to[bend right=30](4,-1);
			\draw [line width=1.5pt,color=blue,step=.5cm,] (4,-1)to[bend right=30](5,1);
			\draw [line width=1.5pt,color=blue,step=.5cm,] (5,1)to[bend right=30] (3,2.5);
			

			\node [draw,circle  ,fill=black,   text=ffffff, font=\huge, inner sep=0pt,minimum size=5mm] (3)  at (3,2.5)  {\scalebox{.5}{$u_{1}$}};	
			\node [draw,circle  ,fill=black,   text=ffffff, font=\huge, inner sep=0pt,minimum size=5mm] (3)  at (1,1)  {\scalebox{.5}{$u_{2}$}};
			\node [draw,circle  ,fill=black,   text=ffffff, font=\huge, inner sep=0pt,minimum size=5mm] (3)  at (2,-1)  {\scalebox{.5}{$u_{3}$}};
			\node [draw,circle  ,fill=ffffff,   text=black, font=\huge, inner sep=0pt,minimum size=5mm] (3)  at (4,-1)  {\scalebox{.5}{$u_{4}$}};
   \node [draw,circle  ,fill=ffffff,   text=black, font=\huge, inner sep=0pt,minimum size=5mm] (3)  at (5,1)  {\scalebox{.5}{$u_{5}$}};
		\end{tikzpicture}
  \caption {In this graph $C_5$, $CD_f=\{u_1, u_2, u_3\}$ is a connected dominating as well as zero forcing set and no subset of the vertex set with cardinality less than three has this property so $F_{cd}(C_5)=3$}
  \end{figure}
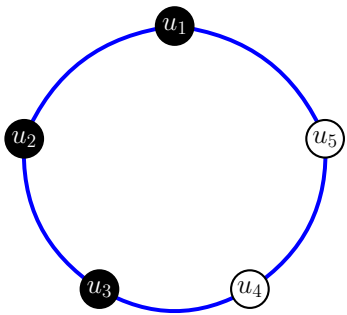
  \end{center}
  
  The focus of this article is on determining the connected dom-forcing number, $F_{cd}(G)$, for various well-known graphs and their specific products, including the corona product, rooted product, and Cartesian product of graphs. This investigation aims to explore the structural properties of $F_{cd}(G)$ and provide new insights into its applications in graph theory.

  The definition makes it evident that the combination of a connected zero forcing set and a connected dominating set constitutes a dom-forcing set. Therefore, the following relationship holds.
\begin{prop} \label{bound}
For any connected graph $G$\\ \begin{center} i) $ Z_c(G) \leq F_{cd}(G)\leq Z_c(G)+\gamma_c (G)$\\
ii) $ \gamma_c(G) \leq F_{cd}(G)\leq Z_c(G)+\gamma_c (G)$
\end{center}
\end{prop}
Also from the definition every connected dom-forcing set is a dom-forcing set and if a dom-forcing set is connected in components, then it is a connected dom-forcing set. Hence we have the following result.
\begin{prop}
For any connected graph $G$,  $$F_{d}(G)\leq F_{cd}(G).$$
\end{prop}
A connected acyclic spanning sub graph of a graph is called spanning tree of a graph. A vertex in a spanning tree is called a leaf if it has degree one.  A maximum leaf spanning tree is a spanning tree that has the largest possible number of leaves among all spanning trees of G. The max leaf number of G is the number of leaves in the maximum leaf spanning tree \cite{leaf}.

If $d$ is the connected domination number of an $n$-vertex graph $G$, where $n > 2$, and $l$ is its max leaf number, then the three quantities $d, l$, and $n$ obey the simple equation $n=d+l$ \cite{mleaf}. Therefore by proposition \ref{bound}, we have $n-l\leq F_{cd}(G).$

There are graphs whose connected domination number is equal to the connected dom-forcing number. From the relation connecting connected domination number and max leaf, we get the following result.
\begin{prop}
    Let $P_n$, $n\geq 4$ be a path of order $n$. Then $$F_{cd}(P_n)=\gamma _c(P_n)=n-2.$$
\end{prop}

\begin{prop}
    Let $C_n$, $n\geq 4$ be a cycle of order $n$. Then $$F_{cd}(C_n)=\gamma _c(C_n)=n-2.$$
\end{prop}
 \begin{prop} \cite{df1} \label{df}
     For any graph G, $F_d(G)=1$ if and only if $G= P_1$ or $  P_2$.
 \end{prop}
  \begin{prop}
    Let $G$ be any connected graph with order greater than two. Then  $F_{cd}(G)\geq2.$
\end{prop}
\begin{proof}
    If possible assume that $F_{cd}(G)=1$. Then by proposition \ref{df} $G$ must be $P_1$ or $P_2$, both have order less than or equal to two. Hence the result.
\end{proof}
From the definition of connected dom-forcing number, dom-forcing number and results in \cite{df1} we can easily verify the following results.
\begin{prop}.
    \begin{itemize}
        \item  For a complete graph $K_n$, $F_d(K_n)=F_{cd}(K_n)=n-1$.
        \item For a complete bipartite graph $K_{m,n}$, $F_d(K_{m,n})=F_{cd}(K_{m,n})=m+n-2$, where $m,n \geq 2$.
        \item For the star graph $K_{1,n}, F_d(K_{1,n}) = F_{cd}(K_{1,n})=n$, for $n \geq 2$.
        \item For  the Petersen graph P, $F_d(P)=F_{cd}(P)=5$
        \item For a wheel graph $W_n$,$F_d(W_n)=F_{cd}(W_n)=3$. (A wheel graph, $W_n$, is a graph obtained by connecting a single vertex to all vertices of a cycle graph $C_{n-1}$.)
        \item For a hyper cube graph $Q_k$, $F_d(Q_k)=F_{cd}(Q_k)=2^{k-1}$.(A hypercube of dimension $k$, $Q_k$, has vertex set $\{0, 1\}^k$, with vertices adjacent when they differ in exactly one coordinate).
        \item For a Coconut tree graph $CT(m,n)$, $F_{cd}[CT(m,n)]=m+n-2$.(A Coconut tree graph $CT(m,n)$ is the graph obtained from the path $P_m$ by appending `n' new pendant edges at an end vertex of $P_m$.)
        \item For a Helm graph $H_m$, $F_{cd}(H_m)=m+1$.(A helm graph is a graph that is created by attaching a pendant edge to each vertex of an n-wheel graph's cycle, and is denoted by $H_m$, where $m \geq 4$).
    \end{itemize}
\end{prop}
The connected dom-forcing number of a graph and its sub graphs are not comparable. Consider the following theorems.
\begin{prop}
    Let $G$ be a graph. If $H$ is a sub graph of $G$, then $F_{cd}(H)\leq F_{cd}(G)$ is not true in general.

\end{prop}
\begin{proof}
    Consider the wheel graph $G=W_{16}$ given below
    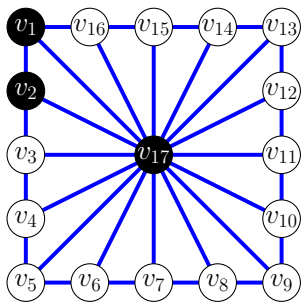
\begin{figure}[H]
\definecolor{ffffff}{rgb}{1,1,1}
\begin{tikzpicture}[scale=0.85]
\draw[line width=1.5pt,color=blue,step=.5cm,] (-1,6)-- (-1,5);
\draw[line width=1.5pt,color=blue,step=.5cm,] (-1,5)-- (-1,4);
\draw[line width=1.5pt,color=blue,step=.5cm,] (-1,4)-- (-1,3);
\draw[line width=1.5pt,color=blue,step=.5cm,] (-1,3)-- (-1,2);
\draw[line width=1.5pt,color=blue,step=.5cm,] (-1,2)-- (0,2);
\draw[line width=1.5pt,color=blue,step=.5cm,] (0, 2)-- (1,2);
\draw[line width=1.5pt,color=blue,step=.5cm,] (1,2)-- (2,2);
\draw[line width=1.5pt,color=blue,step=.5cm,] (2,2)-- (3,2);
\draw[line width=1.5pt,color=blue,step=.5cm,] (3,2)-- (3,3);
\draw[line width=1.5pt,color=blue,step=.5cm,] (3,3)-- (3,4);
\draw[line width=1.5pt,color=blue,step=.5cm,] (3,4)-- (3,5);
\draw[line width=1.5pt,color=blue,step=.5cm,] (3,5)-- (3,6);
\draw[line width=1.5pt,color=blue,step=.5cm,] (3,6)-- (2,6);
\draw[line width=1.5pt,color=blue,step=.5cm,] (2,6)-- (1,6);
\draw[line width=1.5pt,color=blue,step=.5cm,] (1,6)-- (0,6);
\draw[line width=1.5pt,color=blue,step=.5cm,] (0,6)-- (-1,6);

\draw[line width=1.5pt,color=blue,step=.5cm,] (-1,6)-- (1,4);
\draw[line width=1.5pt,color=blue,step=.5cm,] (-1,5)-- (1,4);
\draw[line width=1.5pt,color=blue,step=.5cm,] (-1,4)-- (1,4);
\draw[line width=1.5pt,color=blue,step=.5cm,] (-1,3)-- (1,4);
\draw[line width=1.5pt,color=blue,step=.5cm,] (-1,2)-- (1,4);
\draw[line width=1.5pt,color=blue,step=.5cm,] (0, 2)-- (1,4);
\draw[line width=1.5pt,color=blue,step=.5cm,] (1,2)-- (1,4);
\draw[line width=1.5pt,color=blue,step=.5cm,] (2,2)-- (1,4);
\draw[line width=1.5pt,color=blue,step=.5cm,] (3,2)-- (1,4);
\draw[line width=1.5pt,color=blue,step=.5cm,] (3,3)-- (1,4);
\draw[line width=1.5pt,color=blue,step=.5cm,] (3,4)-- (1,4);
\draw[line width=1.5pt,color=blue,step=.5cm,] (3,5)-- (1,4);
\draw[line width=1.5pt,color=blue,step=.5cm,] (3,6)-- (1,4);
\draw[line width=1.5pt,color=blue,step=.5cm,] (2,6)-- (1,4);
\draw[line width=1.5pt,color=blue,step=.5cm,] (1,6)-- (1,4);
\draw[line width=1.5pt,color=blue,step=.5cm,] (0,6)-- (1,4);

\node [draw,circle  ,fill=black,   text=ffffff, font=\huge, inner sep=0pt,minimum size=5mm] (3)  at (-1,6)  {\scalebox{.5}{$v_1$}};
\node [draw,circle  ,fill=black,   text=ffffff, font=\huge, inner sep=0pt,minimum size=5mm] (3)  at (-1,5)  {\scalebox{.5}{$v_2$}};
\node [draw,circle  ,fill=black,   text=ffffff, font=\huge, inner sep=0pt,minimum size=5mm] (3)  at (1,4)  {\scalebox{.5}{$v_{17}$}};
\node [draw,circle  ,fill=ffffff,   text=black, font=\huge, inner sep=0pt,minimum size=5mm] (3)  at (-1,4)  {\scalebox{.5}{$v_3$}};
\node [draw,circle  ,fill=ffffff,   text=black, font=\huge, inner sep=0pt,minimum size=5mm] (3)  at (-1,3)  {\scalebox{.5}{$v_4$}};
\node [draw,circle  ,fill=ffffff,   text=black, font=\huge, inner sep=0pt,minimum size=5mm] (3)  at (-1,2)  {\scalebox{.5}{$v_5$}};
\node [draw,circle  ,fill=ffffff,   text=black, font=\huge, inner sep=0pt,minimum size=5mm] (3)  at (0,2)  {\scalebox{.5}{$v_6$}};
\node [draw,circle  ,fill=ffffff,   text=black, font=\huge, inner sep=0pt,minimum size=5mm] (3)  at (1,2)  {\scalebox{.5}{$v_7$}};
\node [draw,circle  ,fill=ffffff,   text=black, font=\huge, inner sep=0pt,minimum size=5mm] (3)  at (2,2)  {\scalebox{.5}{$v_8$}};
\node [draw,circle  ,fill=ffffff,   text=black, font=\huge, inner sep=0pt,minimum size=5mm] (3)  at (3,2)  {\scalebox{.5}{$v_{9}$}};
\node [draw,circle  ,fill=ffffff,   text=black, font=\huge, inner sep=0pt,minimum size=5mm] (3)  at (3,3)  {\scalebox{.5}{$v_{10}$}};
\node [draw,circle  ,fill=ffffff,   text=black, font=\huge, inner sep=0pt,minimum size=5mm] (3)  at (3,4)  {\scalebox{.5}{$v_{11}$}};
\node [draw,circle  ,fill=ffffff,   text=black, font=\huge, inner sep=0pt,minimum size=5mm] (3)  at (3,5)  {\scalebox{.5}{$v_{12}$}};
\node [draw,circle  ,fill=ffffff,   text=black, font=\huge, inner sep=0pt,minimum size=5mm] (3)  at (3,6)  {\scalebox{.5}{$v_{13}$}};
\node [draw,circle  ,fill=ffffff,   text=black, font=\huge, inner sep=0pt,minimum size=5mm] (3)  at (2,6)  {\scalebox{.5}{$v_{14}$}};
\node [draw,circle  ,fill=ffffff,   text=black, font=\huge, inner sep=0pt,minimum size=5mm] (3)  at (1,6)  {\scalebox{.5}{$v_{15}$}};
\node [draw,circle  ,fill=ffffff,   text=black, font=\huge, inner sep=0pt,minimum size=5mm] (3)  at (0,6)  {\scalebox{.5}{$v_{16}$}};

\end{tikzpicture}
\caption{The wheel graph $W_{16}$}

\end{figure}

Let $S_1=\{v_1, v_2, v_{17}\}$. Then $S_1$ is the minimum connected dom-forcing set. Hence $F_{cd}(G)=3$. Now consider a subgraph $H=C_{16}$ of $G$ given below.
    \begin{figure}[h]
\definecolor{ffffff}{rgb}{1,1,1}
\begin{tikzpicture}[scale=0.85]
\draw[line width=1.5pt,color=blue,step=.5cm,] (-1,6)-- (-1,5);
\draw[line width=1.5pt,color=blue,step=.5cm,] (-1,5)-- (-1,4);
\draw[line width=1.5pt,color=blue,step=.5cm,] (-1,4)-- (-1,3);
\draw[line width=1.5pt,color=blue,step=.5cm,] (-1,3)-- (-1,2);
\draw[line width=1.5pt,color=blue,step=.5cm,] (-1,2)-- (0,2);
\draw[line width=1.5pt,color=blue,step=.5cm,] (0, 2)-- (1,2);
\draw[line width=1.5pt,color=blue,step=.5cm,] (1,2)-- (2,2);
\draw[line width=1.5pt,color=blue,step=.5cm,] (2,2)-- (3,2);
\draw[line width=1.5pt,color=blue,step=.5cm,] (3,2)-- (3,3);
\draw[line width=1.5pt,color=blue,step=.5cm,] (3,3)-- (3,4);
\draw[line width=1.5pt,color=blue,step=.5cm,] (3,4)-- (3,5);
\draw[line width=1.5pt,color=blue,step=.5cm,] (3,5)-- (3,6);
\draw[line width=1.5pt,color=blue,step=.5cm,] (3,6)-- (2,6);
\draw[line width=1.5pt,color=blue,step=.5cm,] (2,6)-- (1,6);
\draw[line width=1.5pt,color=blue,step=.5cm,] (1,6)-- (0,6);
\draw[line width=1.5pt,color=blue,step=.5cm,] (0,6)-- (-1,6);

\node [draw,circle  ,fill=black,   text=ffffff, font=\huge, inner sep=0pt,minimum size=5mm] (3)  at (-1,6)  {\scalebox{.5}{$v_1$}};
\node [draw,circle  ,fill=black,   text=ffffff, font=\huge, inner sep=0pt,minimum size=5mm] (3)  at (-1,5)  {\scalebox{.5}{$v_2$}};

\node [draw,circle  ,fill=black,   text=ffffff, font=\huge, inner sep=0pt,minimum size=5mm] (3)  at (-1,4)  {\scalebox{.5}{$v_3$}};
\node [draw,circle  ,fill=black,   text=ffffff, font=\huge, inner sep=0pt,minimum size=5mm] (3)  at (-1,3)  {\scalebox{.5}{$v_4$}};
\node [draw,circle  ,fill=black,   text=ffffff, font=\huge, inner sep=0pt,minimum size=5mm] (3)  at (-1,2)  {\scalebox{.5}{$v_5$}};
\node [draw,circle  ,fill=black,   text=ffffff, font=\huge, inner sep=0pt,minimum size=5mm] (3)  at (0,2)  {\scalebox{.5}{$v_6$}};
\node [draw,circle  ,fill=black,   text=ffffff, font=\huge, inner sep=0pt,minimum size=5mm] (3)  at (1,2)  {\scalebox{.5}{$v_7$}};
\node [draw,circle  ,fill=black,   text=ffffff, font=\huge, inner sep=0pt,minimum size=5mm] (3)  at (2,2)  {\scalebox{.5}{$v_8$}};
\node [draw,circle  ,fill=black,   text=ffffff, font=\huge, inner sep=0pt,minimum size=5mm] (3)  at (3,2)  {\scalebox{.5}{$v_{9}$}};
\node [draw,circle  ,fill=black,   text=ffffff, font=\huge, inner sep=0pt,minimum size=5mm] (3)  at (3,3)  {\scalebox{.5}{$v_{10}$}};
\node [draw,circle  ,fill=black,   text=ffffff, font=\huge, inner sep=0pt,minimum size=5mm] (3)  at (3,4)  {\scalebox{.5}{$v_{11}$}};
\node [draw,circle  ,fill=black,   text=ffffff, font=\huge, inner sep=0pt,minimum size=5mm] (3)  at (3,5)  {\scalebox{.5}{$v_{12}$}};
\node [draw,circle  ,fill=black,   text=ffffff, font=\huge, inner sep=0pt,minimum size=5mm] (3)  at (3,6)  {\scalebox{.5}{$v_{13}$}};
\node [draw,circle  ,fill=black,   text=ffffff, font=\huge, inner sep=0pt,minimum size=5mm] (3)  at (2,6)  {\scalebox{.5}{$v_{14}$}};
\node [draw,circle  ,fill=ffffff,   text=black, font=\huge, inner sep=0pt,minimum size=5mm] (3)  at (1,6)  {\scalebox{.5}{$v_{15}$}};
\node [draw,circle  ,fill=ffffff,   text=black, font=\huge, inner sep=0pt,minimum size=5mm] (3)  at (0,6)  {\scalebox{.5}{$v_{16}$}};

\end{tikzpicture}
\caption{The Cycle $C_{16}$}

\end{figure}
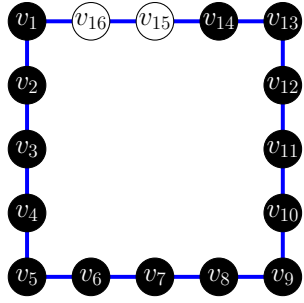

Let $S_2=\{v_1, v_2,\cdots, v_{14}\}$. Then $S_2$ is the minimum connected dom-forcing set. Hence $F_{cd}(H)=14$.Therefore, the assertion follows.
\end{proof}
\begin{prop}
    Let $G$ be a graph. If $H$ is a sub graph of $G$, then $F_{cd}(H)\geq F_{cd}(G)$ is not true in general.

\end{prop}
\begin{proof}
 Consider the  graph $G$ given below

 \begin{figure}[h]
\definecolor{ffffff}{rgb}{1,1,1}
\begin{tikzpicture}[scale=1.5]
    
 \draw[line width=1.5pt,color=blue,step=.5cm,] (-1,0.75) -- (0,1);   
 \draw[line width=1.5pt,color=blue,step=.5cm,] (-0.25,0) -- (0,1); 
 \draw[line width=1.5pt,color=blue,step=.5cm,] (0,1) -- (1.5,1.5);
 \draw[line width=1.5pt,color=blue,step=.5cm,] (2.75,1.25) -- (1.5,1.5);
 \draw[line width=1.5pt,color=blue,step=.5cm,] (2.75,2) -- (1.5,1.5);
 \draw[line width=1.5pt,color=blue,step=.5cm,] (2.25,2.75) -- (1.5,1.5);
 
\node [draw,circle  ,fill=ffffff,   text=black, font=\huge, inner sep=0pt,minimum size=4mm] (3)  at (-1,0.75)  {\scalebox{.4}{$1 $}};
\node [draw,circle  ,fill=black,   text=ffffff, font=\huge, inner sep=0pt,minimum size=4mm] (3)  at (0,1)  {\scalebox{.4}{$3 $}};
\node [draw,circle  ,fill=black,   text=ffffff, font=\huge, inner sep=0pt,minimum size=4mm] (3)  at (-0.25,0)  {\scalebox{.4}{$2 $}};
\node [draw,circle  ,fill=black,   text=ffffff, font=\huge, inner sep=0pt,minimum size=4mm] (3)  at (1.5,1.5)  {\scalebox{.4}{$4 $}};
\node [draw,circle  ,fill=black,   text=ffffff, font=\huge, inner sep=0pt,minimum size=4mm] (3)  at (2.75,1.25)  {\scalebox{.4}{$7 $}};
\node [draw,circle  ,fill=ffffff,   text=black, font=\huge, inner sep=0pt,minimum size=4mm] (3)  at (2.75,2)  {\scalebox{.4}{$6 $}};
\node [draw,circle  ,fill=black,   text=ffffff, font=\huge, inner sep=0pt,minimum size=4mm] (3)  at (2.25,2.75)  {\scalebox{.4}{$5 $}};

\end{tikzpicture}
\caption {The Graph G} 

\end {figure}
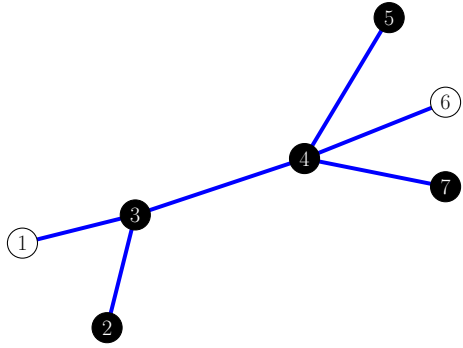

Let $S_1=\{2,3,4,5,7\}$. Then $S_1$ is the minimum connected dom-forcing set. Hence $F_{cd}(G)=5$. Now consider a subgraph $H$ of $G$ given below.
 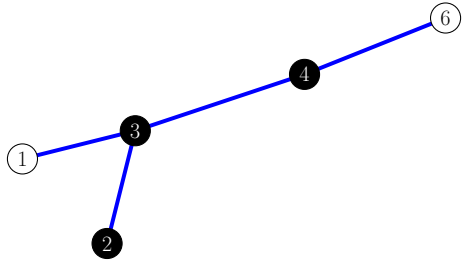
\begin{figure}[h]
\definecolor{ffffff}{rgb}{1,1,1}
\begin{tikzpicture}[scale=1.5]
    
 \draw[line width=1.5pt,color=blue,step=.5cm,] (-1,0.75) -- (0,1);   
 \draw[line width=1.5pt,color=blue,step=.5cm,] (-0.25,0) -- (0,1); 
 \draw[line width=1.5pt,color=blue,step=.5cm,] (0,1) -- (1.5,1.5);
 
 \draw[line width=1.5pt,color=blue,step=.5cm,] (2.75,2) -- (1.5,1.5);

\node [draw,circle  ,fill=ffffff,   text=black, font=\huge, inner sep=0pt,minimum size=4mm] (3)  at (-1,0.75)  {\scalebox{.4}{$1 $}};
\node [draw,circle  ,fill=black,   text=ffffff, font=\huge, inner sep=0pt,minimum size=4mm] (3)  at (0,1)  {\scalebox{.4}{$3 $}};
\node [draw,circle  ,fill=black,   text=ffffff, font=\huge, inner sep=0pt,minimum size=4mm] (3)  at (-0.25,0)  {\scalebox{.4}{$2 $}};
\node [draw,circle  ,fill=black,   text=ffffff, font=\huge, inner sep=0pt,minimum size=4mm] (3)  at (1.5,1.5)  {\scalebox{.4}{$4 $}};

\node [draw,circle  ,fill=ffffff,   text=black, font=\huge, inner sep=0pt,minimum size=4mm] (3)  at (2.75,2)  {\scalebox{.4}{$6 $}};

\end{tikzpicture}
\caption {The Graph H} 

\end {figure}

Let $S_2=\{2,3,4\}$. Then $S_2$ is the minimum connected dom-forcing set. Hence $F_{cd}(H)=3$.Therefore, the assertion follows.
\end{proof}

For a tree  $T$, $F_{cd}(T)\geq F_{cd}(T_2)$. where $T_2$ is any sub tree of $T$.

\begin{prop}
    Let $T$ be a tree with $n\geq  3$ vertices and $m$ leaves. Among these $m $ leaves $n_i$ leaves are adjacent to a vertex $v_i$ of $T$, for $i=1,2, \cdots, r$. Then the connected dom-forcing number of $T$ is $n-r$.
\end{prop}
\begin{proof}
    Given $T$ is a graph with $n$ vertices and $m$ leaves, hence every connected domination set contains all $n-m$ non-leaf vertices. Also, the vertex $v_1$ is adjacent to $n_1$ leaves, to force all these $n_1$ leaves, we must color $n_1-1$ leaves black initially. Therefore every connected dom-forcing set must contain $n-m+(n_1-1)+(n_2-1)+\cdots +(n_r-1)$ vertices. But $n_1+\cdots+n_r=m$, hence $$F_{cd}(T)=n-r.$$
\end{proof}
The above theorem is illustrated in figure \ref{cdt}.
\begin{figure}[h]
\definecolor{ffffff}{rgb}{1,1,1}
\begin{tikzpicture}[scale=1.5]

 \foreach \y in {0,...,4} {\draw[line width=1.5pt,color=blue,step=.5cm,] (\y,1) -- (\y+1,1);}
   
 \draw[line width=1.5pt,color=blue,step=.5cm,] (0,2) -- (1,1);   
 \draw[line width=1.5pt,color=blue,step=.5cm,] (0,0) -- (1,1); 
 \draw[line width=1.5pt,color=blue,step=.5cm,] (3,2) -- (3,1);
 \draw[line width=1.5pt,color=blue,step=.5cm,] (3,0) -- (3,1);
 \draw[line width=1.5pt,color=blue,step=.5cm,] (3,0) -- (2,0);
 \draw[line width=1.5pt,color=blue,step=.5cm,] (3,0) -- (4,0);
 \draw[line width=1.5pt,color=blue,step=.5cm,] (4,2) -- (3,2);
 \draw[line width=1.5pt,color=blue,step=.5cm,] (5,2) -- (4,1);
 \foreach \y in {0,...,5} {
      \node[draw, circle, fill=black,minimum size=4pt] (\y,1) at (\y,1) { };
      }
\node [draw,circle  ,fill=ffffff,   text=black, font=\huge, inner sep=0pt,minimum size=4mm] (3)  at (0,2)  {\scalebox{.4}{$v_1 $}};
\node [draw,circle  ,fill=black,   text=ffffff, font=\huge, inner sep=0pt,minimum size=4mm] (3)  at (0,1)  {\scalebox{.4}{$v_2 $}};
\node [draw,circle  ,fill=black,   text=ffffff, font=\huge, inner sep=0pt,minimum size=4mm] (3)  at (0,0)  {\scalebox{.4}{$v_3 $}};
\node [draw,circle  ,fill=black,   text=ffffff, font=\huge, inner sep=0pt,minimum size=4mm] (3)  at (1,1)  {\scalebox{.4}{$v_4 $}};
\node [draw,circle  ,fill=black,   text=ffffff, font=\huge, inner sep=0pt,minimum size=4mm] (3)  at (2,1)  {\scalebox{.4}{$v_5 $}};
\node [draw,circle  ,fill=black,   text=ffffff, font=\huge, inner sep=0pt,minimum size=4mm] (3)  at (3,1)  {\scalebox{.4}{$v_6 $}};
\node [draw,circle  ,fill=black,   text=ffffff, font=\huge, inner sep=0pt,minimum size=4mm] (3)  at (3,2)  {\scalebox{.4}{$v_7 $}};
\node [draw,circle  ,fill=ffffff,   text=black, font=\huge, inner sep=0pt,minimum size=4mm] (3)  at (4,2)  {\scalebox{.4}{$v_8 $}};
\node [draw,circle  ,fill=black,   text=ffffff, font=\huge, inner sep=0pt,minimum size=4mm] (3)  at (3,0)  {\scalebox{.4}{$v_9 $}};
\node [draw,circle  ,fill=black,   text=ffffff, font=\huge, inner sep=0pt,minimum size=4mm] (3)  at (2,0)  {\scalebox{.4}{$v_{10} $}};
\node [draw,circle  ,fill=ffffff,   text=black, font=\huge, inner sep=0pt,minimum size=4mm] (3)  at (4,0)  {\scalebox{.4}{$v_{11} $}};
\node [draw,circle  ,fill=ffffff,   text=black, font=\huge, inner sep=0pt,minimum size=4mm] (3)  at (5,2)  {\scalebox{.4}{$v_{13} $}};
\node [draw,circle  ,fill=black,   text=ffffff, font=\huge, inner sep=0pt,minimum size=4mm] (3)  at (5,1)  {\scalebox{.4}{$v_{14} $}};
\node [draw,circle  ,fill=black,   text=ffffff, font=\huge, inner sep=0pt,minimum size=4mm] (3)  at (4,1)  {\scalebox{.4}{$v_{12} $}};
\end{tikzpicture}
\caption {A tree $T$ with 14 vertices and 8 leaves. Among these leaves $v_1, v_2, v_3$ are adjacent to $v_4$; $v_8$ is adjacent to $v_7$; $v_{10}, v_{11}$ are adjacent to $v_9$ and $v_{13}, v_{14}$ are adjacent to $v_{12}$. Hence $F_{cd}(T)=14-4=10$.} 
\label{cdt}
\end {figure}
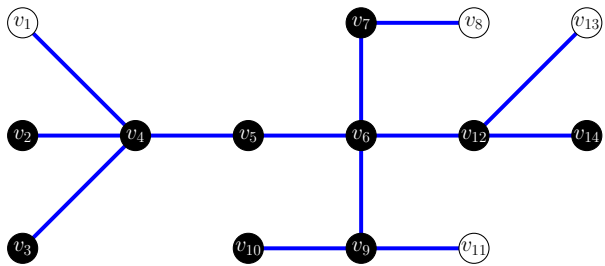

The join of two graphs is their union with all the edges $gh$ where $g \in V (G)$ and $h \in V (H)$. The join of graphs G and H is denoted by $G\vee H$. Let us check the connected dom-forcing number of $G\vee H$.

 \begin {prop}
 Let $G$ and $H$ be two connected graphs. Then
$$F_{cd}(G \vee H)=F_d(G \vee H)=Z(G \vee H) = min\{|H| + Z(G); |G| + Z(H)\}.$$
\end {prop}
\begin{proof}
    Every zero forcing set dominates $G\vee H$, since by the definition of join, one vertex from both G and H dominates $G\vee H$. Also it is connected. Therefore $$ F_{cd}(G \vee H)=F_d(G \vee H)=Z(G \vee H) = min\{|H| + Z(G); |G| + Z(H)\}.$$
\end{proof}
\section{Connected dom-forcing number of product of graphs }
In this section we discuss the connected dom-forcing number of different type of product of graphs, we start with the corona product.

The corona of $G$ with $H$, denoted by $G\odot H$, is the graph of order $|V (G)||V (H)| +|V (G)|$ obtained by taking one copy of $G$ and $|V (G)|$ copies of $H$, and joining all vertices in the $i^{th}$ copy of $H$ to the i-th vertex of $G$ \cite{zpr}.

Let G be a graph with $|V(G)| = n$, where $V(G)$ denotes the vertex set of $G$, and let $H_1, \cdots, H_n$ be the graphs. Define the graph $G\prec H_1, \cdots, H_n \succ $ as the graph obtained by joining all the vertices of the graph $H_i$ to the i-th vertex of G, where $i = 1, \cdots , n$. Note that $|H_i|$ can be zero, in which case no extra vertices will be joined to the vertex i. We call the graph $G\prec H_1, \cdots, H_n \succ $ a generalized corona of G with $H_1, \cdots, H_n$ \cite{z5}.
\begin{theorem} \cite{czero}
    For graph G with  $|V(G)| = n$ and  $H_1, \cdots, H_n$ be graphs with order at lest two, we have $$Z_c(G\prec H_1, \cdots, H_n \succ)= |V(G)|+\sum_{i=1}^{n} Z(H_i) .$$ 
\end{theorem}
\begin{theorem} 
    For graph G with  $|V(G)| = n$ and  $H_1, \cdots, H_n$ be graphs with order at lest two, we have $$F_{cd}(G\prec H_1, \cdots, H_n \succ) =Z_c(G\prec H_1, \cdots, H_n \succ)= |V(G)|+\sum_{i=1}^{n} Z(H_i) .$$ 
\end{theorem}
\begin{proof}
    Let $Z_i$ be a minimum zero forcing set of $H_i$ for all $1 \leq i \leq n$. We can easily seen that, $V(G)\cup Z_1 \cup \cdots \cup Z_n$ is a minimum connected zero forcing set of the $G\prec H_1, \cdots, H_n \succ $ . From the definition of generalised corona product, this set dominates $G\prec H_1, \cdots, H_n \succ $. Hence $$F_{cd}(G\prec H_1, \cdots, H_n \succ) =Z_c(G\prec H_1, \cdots, H_n \succ)= |V(G)|+\sum_{i=1}^{n} Z(H_i) .$$ 
\end{proof}
For graph G with  $|V(G)| = n$ and for any graph $H$, from the definition we can see that $G\odot H= G\prec H, \cdots, H \succ$. Therefore from the above theorem we can easily derive the inequality.
\begin{corollary}\label{corona}
    For graph G with  $|V(G)| = n$ and for any graph $H$ with order at least two, $$F_{cd}(G\odot H)= n(1+ Z(H)).$$
\end{corollary}
Let G and H be two graphs of order $n_1$ and $n_2$ respectively. For any integer $k \geq 2$, we define the graph $ G \odot ^k H $ recursively from $G\odot H$ as $G\odot ^k H = (G \odot ^{k-1} H)\odot H$. It is also noted that $|G\odot ^{k-1}H| = n_1(n_2 +1)^{k-1}$ and $|G\odot^k H| = |G\odot^{k-1}H|+n_1n_2(n_2+1)^{k-1}$.

By repeated application of Corollary \ref{corona}, we get the following result.

\begin{theorem}
    Let G and H be two graphs of order $n_1$ and $n_2,(n_2\geq 2),$ respectively. Then $$F_{cd}(G\odot^k H)= n_1(n_2 +1)^{k-1}(Z(H)+1).$$
\end{theorem}
\begin{theorem} \label{corona2}
    If $G$  is a graph with   $|V(G)| = n$, then $F_{cd}(G\odot K_1)= n.$
\end{theorem}
\begin{proof}
    Let $\{v_1, v_2, \cdots, v_n\}$ be the vertex set of G. Then $A=\{v_1, \cdots, v_n\}$ be a connected dom forcing set for $G\odot K_1$. Therefore $F_{cd} (G\odot K_1)\leq n$.

    Note that every connected dominating set contains all vertices of G. Therefore  $F_{cd}(G\odot K_1)\geq n$.

    From the above we have $F_{cd}(G\odot K_1)= n$.
\end{proof}
By repeated application of Theorem \ref{corona2}, we get the following result
\begin{corollary}
    For graph G with  $|G| = n$, then $F_{cd}(G\odot^k K_1)= 2^{k-1}n.$
\end{corollary}
Let G be a connected graph with vertices $v_1, v_2, \cdots, v_n$ and let $H$ be a sequence of n rooted graphs $H_1,H_2,\cdots,H_n$. The rooted product of G and H is defined as the graph obtained by identifying the root of $H_i, 1 \leq i \leq n$, with the ith vertex of $G$ for all i. This graph is denoted by G(H) and is known as the rooted product of G by H \cite{root}.
\begin{theorem}
    Let G be a graph with order $n$ and $H$ be a graph of order at least two, rooted with any vertex and its connected dom forcing number is $F_{cd}(H)$. Let the rooted vertex is an element of the minimum connected dom-forcing set. Then the connected dom forcing number of the rooted product of $G$ by $H$,$$F_{cd}[G(H)]= n (F_{cd}(H)).$$
\end{theorem}
\begin{proof}
    Let $CD_{f_{H_i}} $ be the minimum connected dom-forcing set for ith copy of $H$. Then $A= \bigcup_{i=1}^nCD_{f_{H_i}}$ form a connected dom-forcing set for rooted product of $G$ by $H$. Hence $$F_{cd}[G(H)]\leq n (F_{cd}(H)).$$
    Given that the rooted vertex is an element of the minimum connected dom-forcing set, every connected dom-forcing set contains $CD_{f_{H_i}} $ for all $i$. Hence $$F_{cd}[G(H)]\geq n (F_{cd}(H)).$$ From these we have $$F_{cd}[G(H)]= n (F_{cd}(H)).$$
\end{proof}

We know The Cartesian product of two graphs $G$ and $H$, is denoted by $G \Box H$, is the graph with vertex set $V (G) \times V (H)$ such that $(u, v)$ is adjacent to $(u', v')$ if and only if \\(1) $u = u'$ and $vv' \in E(H)$, or \\ (2) $v = v'$ and $uu' \in E(G)$.\\ The Cartesian product of two path graphs is known as the grid graph.

Let $p$ and $q$ be positive integers such that $p \leq q$. A $p \times q$ grid graph \cite{lie} $G_{p, q} = (V, E)$  is a graph where
$$V = \{(i, j)|1 \leq i \leq p, 1 \leq j \leq q\},$$
$$E = \{\{(i, j),(i, j + 1)\}|1 \leq i \leq p, 1 \leq j \leq q - 1\}$$
$$\hspace{1.25cm}\cup \{\{(i, j),(i + 1, j)\}|1 \leq i \leq p - 1, 1 \leq j \leq q\}.$$

The ladder graph $L_n$ is obtained by taking the cartesian product of path $P_n$ with the complete graph $K_2$.
\begin{theorem}\cite{lie} \label{c1}
    Connected domination number of the Ladder graph $L_n$ is 
   $$\gamma _c (L_n)=\gamma _c(G_{2,n})=\left \{ \begin{array}{ccc}
       2  &if & n=2,3 \\
       
       n  &if & n\geq 4\\
    \end{array} \right .$$
\end{theorem}

\begin{theorem}
    The connected dom-forcing number of the Ladder graph $L_n$ is 
    $$F_{cd}(L_n)=F_{cd}(G_{2,n})=n $$
\end{theorem}
\begin{proof}
    The graph $G_{2,2}$ is cycle $C_4$ and we know that $F_{cd}(C_4)=2$.
    In the case of $G_{2,3}$, no minimum connected domination set cardinality 2 forces $G_{2,3}$. But $A=\{(1,1),(1,2),(1,3)\}$ be a dom-forcing set. Hence $F_{cd}(G_{2,3})=3 $. For $n\geq 4$, $A=\{(1,1),(1,2),\cdots, (1,n)\}$ form a connected dom-forcing set, which is minimum by theorem \ref{c1}. Hence $$F_{cd}(L_n)=F_{cd}(G_{2,n})=n. $$
\end{proof}
\begin{theorem}\cite{lie} \label{c2}
    Connected domination number of $3\times p$ grid graph $G_{3,p}$ is $\gamma _c(G_{3,p})=p$.
\end{theorem}

\begin{theorem}
    The connected dom-forcing number of $3\times p$ grid graph $G_{3, p}$ is 
    $$F_{cd}(G_{3,p})=p+1 $$
\end{theorem}
\begin{proof}
    We can see that  $A=\{(2,1),(2,2),\cdots, (2,p)\}$ be a minimum connected domination set, and  $A$ is unique. The set $A$ cannot forces $G_{3,p}$. But $A\cup \{(1,1)\}$ form a connected dom-forcing set, which is minimum. Hence $$F_{cd}(G_{3,p})=p+1. $$
\end{proof}
\begin{theorem}\cite{cg} \label{g1}
    Connected domination number of $p\times q$ ($p\geq 4, q\geq 4$) grid graph $G_{p,q}$ is
    $$\gamma _c(G_{p,q})=\frac{pq-a'_{p,q}}{3}+\bar{r}'_{p,q}+c'_{p,q}.$$
    in which
    $$a'_{p,q}=p(\text{mod 3} ) . q(\text{mod 3} )$$
    $$\bar{r}'_{p,q}= \left \{ \begin{array}{cl}
       3  & p(\text{mod 3} ) . q(\text{mod 3} )=4   \\
         2  & p(\text{mod 3} ) . q(\text{mod 3} )=2   \\
          1  & p(\text{mod 3} ) . q(\text{mod 3} )=1   \\
           0  & p(\text{mod 3} ) . q(\text{mod 3} )=0   \\
    \end{array}
    \right.$$
       $$c'_{p,q}= \left \{ \begin{array}{ll}
       \min \{ \frac p 3, \frac q 3\}  & p(\text{mod 3} )=0 \text { and } q(\text{mod 3} )=0   \\
        \frac p 3  &  p(\text{mod 3} )=0 \text { and } q(\text{mod 3} )\neq 0    \\
          \frac q 3 & p(\text{mod 3} )\neq 0 \text { and } q(\text{mod 3} )= 0    \\
           \lfloor\frac p 3\rfloor+ \lfloor\frac q 3\rfloor -1 & p(\text{mod 3} )\neq 0 \text { and } q(\text{mod 3} )\neq 0   \\
    \end{array}
    \right.$$
\end{theorem}
\begin{theorem}
        For $p\times q$ ($p\geq 4, q\geq 4$) grid graph $G_{p,q}$ ,
    $$\gamma _c(G_{p,q}) \leq F _{cd}(G_{p,q}) \leq \gamma _c(G_{p,q})+1.$$
\end{theorem}
\begin{proof}
Consider the following cases.\\\\
\textbf{Case 1:} For  any $q$ and p (mod 3) $=0 $. We consider the vertex set $S = A \cup B \cup C$ in which $A = \{(1, 2), (2, 2), \cdots , (p, 2)\}, B = \{(2, 3), (2, 4), \cdots , (2, q)\}$, and $C =\{(x, 3), \cdots , (x, q) | x = 5,8, \cdots ,p-1\}$. Theorem \ref{g1} says that $S$ is a minimum connected dominating set with cardinality $\frac{p(q+1)}{3}$.\\\\
\textbf{Case 2:}  For any $p$ and q (mod 3)$=0$, we consider the vertex set $S = A \cup B \cup C'$ in which $C' = \{(3, y), \cdots , (p, y) | y = 5,8,\cdots, q - 1\}$. Then $S $ is a minimum connected dominating set with cardinality $\frac{(p+1)q}{3}$ (By Theorem \ref{g1}). In both cases  $A\cup \{(1, 1)\}$ forces the entire graph. Hence
$$\gamma _c(G_{p,q}) \leq F _{cd}(G_{p,q}) \leq \gamma _c(G_{p,q})+1.$$
\textbf{Case 3:} For p (mod 3) $=1$ and q (mod 3) $=1$. We consider the vertex set $S = A \cup B_1 \cup C_1 \cup E_1$ for $p=4$ and $S = A \cup B \cup C_1 \cup D_1 \cup E_1$ for $p >4$,  in which  $B_1 = \{(2, 3), (2, 4), \cdots , (2, q-2)\}$, $C_1 = \{(3, y), \cdots , (p, y) | y = 5,8, \cdots, q -2\}$, $D_1 = \{(x, q - 1), (x, q) | x = 5,8, \cdots ,p - 5\}$ and $E_1 = \{(p-1, q - 1),(p-1, q),(p-2, q)\}$. Then $S $ is a minimum connected dominating set with cardinality $\frac{pq+p+q-3}{3}$ (By Theorem \ref{g1}). This set $S$ forces the entire graph, because of the following.\\

Color the vertices in $S$  black. The only white neighbour of $(p-1, q)$ is $(p, q)$. Therefore,  $(p-1, q) \rightarrow (p, q)\rightarrow (p, q-1)$. Again the vertices in $C_1$ and $D_1$ are black, we can force the top three rows by using $C_1$ and $D_1$. Hence we can force the entire graph black. Therefore,
$$\gamma _c(G_{p,q}) = F _{cd}(G_{p,q}).$$\\
\textbf{Case 4:} For p (mod 3) $=1$ and q (mod 3) $=2$. We consider the vertex set $S = A \cup B \cup C_2 \cup D_2 \cup E_2$, in which $C_2 = \{(3, y), \cdots , (p, y) | y = 5,8, \cdots, q - 3\}$, $D_2 = \{(x, q - 2), (x, q -1), (x, q) | x = 5,8, \cdots ,p - 2\}$ and $E_2 = \{(p, q - 2), (p, q - 1)\}$. Then $S $ is a minimum connected dominating set with cardinality $\frac{pq+p+q-2}{3}$ (By Theorem \ref{g1}). This set forces the entire graph, since the following.\\

Color the vertices in $S$  black. The only white neighbour of $(p, q-2)$ is $(p-1, q-2)$, therefore  $(p, q-2) \rightarrow (p-1, q-2)\rightarrow (p-1, q-1)\rightarrow (p-1, q)\rightarrow (p, q) $. Again vertices in $C_2$ and $D_2$ are black, we can forces the top four rows by $C_2$ and $D_2$. Hence we can force the entire graph black. Therefore,
$$\gamma _c(G_{p,q}) = F _{cd}(G_{p,q}) .$$\\
\textbf{Case 5:} For p (mod 3) $=2$ and q (mod 3) $=1$: We consider the vertex set $S = A \cup B \cup C_3 \cup D_3 \cup E_3$, in which $C_3 = \{(3, y), \cdots , (p, y) | y = 5,8, \cdots, q - 2\}$, $D_3 = \{(x, q - 2), (x, q -1), (x, q) | x = 5,8, \cdots ,p - 3\}$ and $E_3 = \{(p, q - 1), (p, q)\}$. Then $S $ is a minimum connected dominating set with cardinality $\frac{pq+p+q-2}{3}$ (By Theorem \ref{g1}). This set forces the entire graph, since the following.

Color the vertices in $S$  black. The only white neighbour of $(p, q-1)$ is $(p-1, q-1)$, and $(p, q)$ is $(p-1, q)$. Therefore  $(p, q-1) \rightarrow (p-1, q-1)\rightarrow (p-2, q-1) $;$(p, q) \rightarrow (p-1, q)\rightarrow (p-2, q)$. Again vertices in $B_3$,$C_3$ and $D_3$ are black, we can forces the top three rows by $B_3$,$C_3$ and $D_3$. Hence we can force the entire graph black. Therefore,
$$\gamma _c(G_{p,q}) = F _{cd}(G_{p,q}) .$$\\\\
\textbf{Case 6:} For p (mod 3) $=2$ and q (mod 3) $=2$. We consider the vertex set $S = A \cup B \cup C_4 \cup D_4 \cup E_4$, in which $C_4 = \{(3, y), \cdots, (p, y) | y = 5,8, \cdots, q - 3\}$, $D_4 = \{(x, q - 2), (x, q -1), (x, q) | x = 5,8, \cdots ,p - 3\}$ and $E_4 = \{(p, q-2),(p, q - 1), (p, q)\}$.Then $S $ is a minimum connected dominating set with cardinality $\frac{pq+p+q-2}{3}$ (By Theorem \ref{g1}). This set forces the entire graph, since the following.\\

Color the vertices in $S$  black. The only white neighbour of $(p, q-2)$ is $(p-1, q-2)$, $(p, q-1)$ is $(p-1, q-1)$ and $(p, q)$ is $(p-1, q)$. Therefore  $(p, q-2) \rightarrow (p-1, q-2)\rightarrow (p-2, q-2) $; $(p, q-1) \rightarrow (p-1, q-1)\rightarrow (p-2, q-1) $;$(p, q) \rightarrow (p-1, q)\rightarrow (p-2, q)$. Again vertices in $B_4$,$C_4$ and $D_4$ are black, we can forces the top four rows by $B_4$,$C_4$ and $D_4$. Hence we can force the entire graph black. Therefore,
$$\gamma _c(G_{p,q}) = F _{cd}(G_{p,q}) .$$

In all cases we have $$\gamma _c(G_{p,q}) \leq F _{cd}(G_{p,q}) \leq \gamma _c(G_{p,q})+1.$$
\end{proof}
\begin{corollary}
    For $p$ and $q$ $\equiv 1(\text{ mod 3})$, connected dom-forcing number of a $p\times q$ grid graph $G_{p,q}$ is $$F_{cd}(G_{p,q})=\frac{pq+p+q-3}{3}.$$
\end{corollary}
\begin{corollary}
    If $p,q$ is not a multiple of three and $p$ or $q$ $\equiv 2(\text{ mod 3})$, then connected dom-forcing number of a $p\times q$ grid graph $G_{p,q}$ is $$F_{cd}(G_{p,q})=\frac{pq+p+q-2}{3}.$$
\end{corollary}

In the case of $G_{11,11}$, the above theorem is illustrated in Figure \ref{cdg}.
\begin{figure}[h]
\definecolor{ffffff}{rgb}{1,1,1}
\begin{tikzpicture}[scale=0.75]

  \foreach \x in {0,...,10} {\foreach \y in {0,...,9} {\draw[line width=1.5pt,color=blue,step=.5cm,] (\x,\y) -- (\x,\y+1);}}
  
  \foreach \x in {0,...,9} {
    \foreach \y in {0,...,10} {
      \draw[line width=1.5pt,color=blue,step=.5cm,] (\x,\y) -- (\x+1,\y);
      }
    }
    \foreach \x in {0,...,10} {
    \foreach \y in {0,...,10} {
      \node[draw, circle, fill=ffffff,minimum size=4pt] (\x,\y) at (\x,\y) { };
      }
    }
    \foreach \y in {0,...,10} {
      \node[draw,circle  ,fill=black,   text=ffffff, font=\huge, inner sep=0pt,minimum size=4mm] (\y,1) at (\y,1) {\scalebox{.4}{$A$}};}
      \foreach \y in {2,...,10} {
      \node[draw,circle  ,fill=black,   text=ffffff, font=\huge, inner sep=0pt,minimum size=4mm] (1,\y) at (1,\y) {\scalebox{.4}{$B$}};}
      \foreach \x in{2,...,10}{\foreach \y in {4,7} {
      \node[draw,circle  ,fill=black,   text=ffffff, font=\huge, inner sep=0pt,minimum size=4mm] (\x,\y) at (\x,\y) {\scalebox{.4}{$C$}};}}
      \foreach \x in{8,9,10}{\foreach \y in {4,7} {
      \node[draw,circle  ,fill=black,   text=ffffff, font=\huge, inner sep=0pt,minimum size=4mm] (\y,\x) at (\y,\x) {\scalebox{.4}{$D$}};}}
      \foreach \x in{8,9,10}{
      \node[draw,circle  ,fill=black,   text=ffffff, font=\huge, inner sep=0pt,minimum size=4mm] (10,\x) at (10,\x) {\scalebox{.4}{$E$}};}
\end{tikzpicture}
\caption {Connected dom-forcing set as well as connected dominating set  for $G_{11,11}$ } 
\label{cdg}
\end {figure}
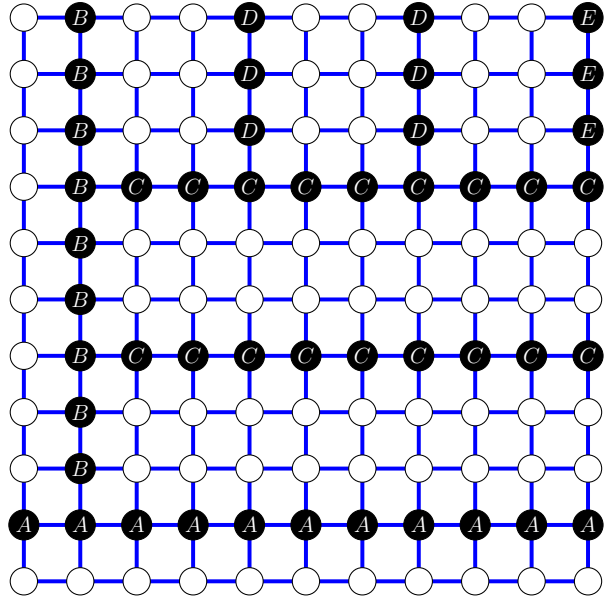
\begin{theorem}
    The connected dom-forcing number of Cartesian product of cycle $C_n$ and path $P_2$ is $n$. ie $$F_{cd}(C_n\Box P_2)=n.$$
\end{theorem}
\begin{proof}
    Let $v_1,\cdots, v_n$ be the vertices of one copy of $C_n$ in $C_n\Box P_2$ and $v'_1,\cdots, v'_n$ be the vertices of other copy of $C_n$ in $C_n\Box P_2$, each $v_i$ is adjacent to $v_i'$. We can see that every connected dominating set contain at least $n$ vertices, and let these vertices be $\{v_1,\cdots, v_n\}$. Also these vertices forces the entire graph black. Hence $$F_{cd}(C_n\Box P_2)=n.$$
\end{proof}
\section{Connected dom-forcing number of splitting graph of a graph $G$ }
The splitting graph of a graph $G$ is the graph $S(G)$ obtained by taking a vertex $v'$ corresponding to each vertex $v \in G$ and join $v'$ to all vertices of $G$ adjacent to $v$ \cite{split}. 
\begin{theorem} \cite{cdsplit} \label{cds}
    Let $G$ be a connected graph of order $n$. Then 
    $$\gamma _c (S(G))=\left \{ \begin{array}{ccc}
       2  &if & \gamma_c(G)=1 \\
      \gamma_c(G) &if & \gamma_c(G)\geq 2\\
    \end{array} \right .$$
\end{theorem}
 
\begin{theorem} \cite{zpr} \label{zsp}
Let $S(P_n)$ be the splitting graph of the path $P_n$. Then $Z[S(P_n)]=2 Z(P_n)=2$.
\end{theorem}
\begin{theorem}
    Let $S(P_n)$ be the splitting graph of the path $P_n$. Then  $$n-1 \leq F_{cd}[S(P_n)]\leq n.$$
\end{theorem}
\begin{proof}
     Let $v_1,v_2, \cdots, v_n $ be the vertices of the path graph $P_n$ in $ S(P_n)$, and let $u_1, u_2, \cdots ,u_n$ be the vertices corresponding to $v_1,v_2, \cdots, v_n$ which are added to obtain $S(P_n)$. Then $A=\{v_2, \cdots, v_{n-1} \}$    be a minimum connected dominating set by theorem \ref{cds} and is unique. We can see that this set is not a zero forcing set. Hence $F_{cd}[S(P_n)]\geq n-1$. The set $B=\{v_1, u_1\}$ zero forces $S(P_n)$ by theorem \ref{zsp}. Therefore $A \cup B$ forms a connected dom-forcing set of $S(P_n)$. Hence $F_{cd}[S(P_n)]\leq n$. From these we have $n-1 \leq F_{cd}[S(P_n)]\leq n$.
\end{proof}
\begin{theorem} \cite{df1}
    Let $S(P_n)$ be the splitting graph of the path $P_n$. Then for $2\leq n \leq 4$, $F_d[(S(P_n)]=n$.
\end{theorem}

From the above two theorems we can easily verify the following result
\begin{corollary}
    Let $S(P_n)$ be the splitting graph of the path $P_n$. Then for $2\leq n \leq 4$, $F_{cd}[(S(P_n)]=n$.
\end{corollary}
\begin{theorem}
    Let $S(P_n)$ be the splitting graph of the path $P_n$. Then for $ n \geq 5$, $F_{cd}[(S(P_n)]=n-1$.
\end{theorem}
\begin{proof}
     Let $v_1,v_2, \cdots, v_n $ be the vertices of the path graph $P_n$ in $ S(P_n)$, and let $u_1, u_2, \cdots ,u_n$ be the vertices corresponding to $v_1,v_2, \cdots, v_n$ which are added to obtain $S(P_n)$. For $S(P_5)$, $CD_f= \{v_2,u_3, v_4, v_5\}$, and for $S(P_6)$, $CD_f= \{v_2,v_3, u_4, v_5, v_6\}$ form a connected dom-forcing set and is minimum. For $n \geq 7 $ consider the  subsets $CD_f=\{ v_2, v_3, u_3, u_4, v_5, \cdots, v_{n-1}\}$ of the vertex set of $S(P_n)$. he set $CD_f$ forms a connected dom-forcing set of the graph $S(P_n)$  and is minimum.  Hence for $ n \geq 7$, $F_{cd}[(S(P_n)]=n-1$.
\end{proof}

\begin{theorem} \cite{ zpr} \label{czc}
Let $S(C_n)$ be the splitting graph of the cycle $C_n$. Then $Z(S(C_n))=2 Z(C_n)=4$. 
\end{theorem}
\begin{theorem}
    Let $S(C_n)$ be the splitting graph of the cycle $C_n$. Then  $$n-1 \leq F_{cd}[S(C_n)]\leq n.$$
\end{theorem}
\begin{proof}
     Let $v_1,v_2, \cdots, v_n $ be the vertices of the cycle graph $C_n$ in $S(C_n)$, and  let $u_1, u_2, \cdots ,u_n$ be the vertices corresponding to $v_1,v_2, \cdots, v_n$ which are added to obtain $S(C_n)$. Then $A=\{v_1, \cdots, v_{n-2} \}$     be a minimum connected dominating set by theorem \ref{cds}. This set is unique up to isomorphism. We can see that this set is not a zero forcing set. Hence $F_{cd}[S(C_n)]\geq n-1$. The set $B=\{v_1, v_2, u_2, u_3\}$ zero forces $S(C_n)$ by theorem \ref{czc}. Therefore $A \cup B$ be a connected dom-forcing set of $S(C_n)$. Hence $F_{cd}[S(C_n)]\leq n$. From these we have $n-1 \leq F_{cd}[S(C_n)]\leq n$.
\end{proof}

\begin{corollary}
    Let $S(C_n)$ be the splitting graph of the cycle $C_n$. Then for $ n \geq 7$, $F_{cd}[(S(C_n)]=n-1$.
\end{corollary}
\begin{proof}
    Let $v_1,v_2, \cdots, v_n $ be the vertices of the cycle graph $C_n$ in $S(C_n)$, and  let $u_1, u_2, \cdots ,u_n$ be the vertices corresponding to $v_1,v_2, \cdots, v_n$ which are added to obtain $S(C_n)$. For $n \geq 7 $ consider the  subsets $CD_f=\{ v_1, v_2, u_2, u_3, v_4, \cdots, v_{n-2}\}$ of the vertex set of $S(C_n)$. The set $CD_f$ forms a connected dom-forcing set of the graph $S(C_n)$  and is minimum.  Hence for $ n \geq 7$, $F_{cd}[(S(C_n)]=n-1$.
\end{proof}

\begin{theorem}\cite{df1}
 Let $S(K_{1,n})$ be the splitting graph of a star graph $K_{1,n}$. Then for $n \geq 2$, $F_d [S(K_{1,n})]= 2n-1$.
\end{theorem}
\begin{theorem}
 Let $S(K_{1,n})$ be the splitting graph of the Star Graph $K_{1,n}$. Then for $n \geq 2$, $F_{cd} [S(K_{1,n})]= 2n-1$.
\end{theorem}
\begin{proof}
For $n\geq 2$, let $u_1, v_1,v_2, \cdots, v_n $ be the vertices of the star graph $K_{1,n}$ with $\deg(u_1)=n$ and $u'_1, v'_1,v'_2, \cdots ,v'_n$ be the vertices corresponding to $u_1, v_1, v_2, \cdots, v_n$ which are added to obtain $S(K_{1,n})$. From the above theorem  \\ $B=\{u_1, v_2, \cdots, v_n, v'_2, \cdots , v'_n \} $ form a dom-forcing set which is connected. Hence $F_{cd} [S(K_{1,n})]= 2n-1$.
\end{proof}

\begin{theorem} \cite{zpr}
 Let $S(L_n)$ be the splitting graph of the ladder graph $L_n$. Then for $n\geq 2$, $Z[S(L_n)]=4$. 
\end{theorem}

\begin{theorem}
 For $n\geq 2$, $$n+1 \leq F_{cd}[S(L_n)]\leq n+2.$$
\end{theorem}
\begin{proof}
Let $\{u_1, u_2, \cdots, u_n, v_1, v_2, \cdots, v_n\}$ be the vertices of the ladder graph $L_n$ and let $\{u'_1, u'_2, \cdots, u'_n, v'_1, v'_2, \cdots, v'_n\}$ be the vertices corresponding to \\$\{u_1, u_2, \cdots, u_n, v_1, v_2, \cdots, v_n\}$ which are added to obtain $S(L_n)$. Then $A=\{v_1, \cdots, v_{n} \}$     be a minimum connected dominating set by theorem \ref{cds}. We can see that no minimum connected dominating set zero forces the entire graph. Hence $F_{cd}[S(L_n)]\geq n+1$. The set $B=\{v_1, v_2, u'_1, v'_1\}$ zero forces $S(L_n)$ by the above theorem. Therefore the set $A \cup B$ is a connected dom-forcing set of $S(L_n)$. Hence $F_{cd}[S(L_n)]\leq n+2$. From these we have  $$n+1 \leq F_d[S(L_n)]\leq n+2.$$
\end{proof}

\section{Summary and open questions}

In this article, we introduced a new parameter in graph theory called the connected dom-forcing number. In Section 2, we established upper and lower bounds for this parameter in terms of the connected zero forcing number and the connected domination number.\\

Section 3 explores the connected dom-forcing number for various types of product graphs, including the corona product, generalized corona product, rooted product, and Cartesian product of graphs.\\

In Section 4, we examined the connected dom-forcing number of the splitting graph of a given graph. We considered the connected dom-forcing number for specific graphs such as $S(P_{n}), S(C_{n}), S(K_{1,n})$ and $S(L_{n})$. \\

Some questions remain open, such as the characterization of graphs where  $ F_{cd}(H) = F_{cd} (G)$, with $H$ being a subgraph of $G$.  Another open problem is the characterization of ladder graphs satisfying $n+1 = F_{cd} [S(L_{n})]$ and $n+2 = F_{cd} [S(L_{n})]$.      
\bibliographystyle{unsrt} 


\end{document}